\documentclass{amsart}
\usepackage[dvipdfmx]{graphicx}
\usepackage[dvipdfmx]{color}
\usepackage{enumerate}

\definecolor{RedOrange}{cmyk}{ 0, 0.77, 0.87, 0}
\definecolor{RoyalPurple}{cmyk}{ 0.84, 0.53, 0, 0}
\definecolor{YellowGreen}{cmyk}{ 0.44, 0, 0.74, 0}
\definecolor{Fuchsia}{cmyk}{ 0.47, 0.91, 0, 0.08}
\definecolor{Blue}{cmyk}{ 0.84, 0.53, 0, 0}
\definecolor{BlueViolet}{cmyk}{ 0.84, 0.53, 0, 0}
\definecolor{Black}{cmyk}{ 0.75, 0.68, 0.67, 0.9}
\usepackage{xcolor}
\usepackage{verbatim}
\usepackage{amsmath}
\usepackage{amssymb, mathrsfs}
\usepackage{amsbsy}
\usepackage{amscd}
\usepackage{amsthm}
\usepackage{bm}

\newcommand{\N}{\mathbb{N}}
\newcommand{\R}{\mathbb{R}}
\newcommand{\Z}{\mathbb{Z}}

\newcommand{\G}{\mathbb{G}}
\newcommand{\E}{\mathbb{E}}

\newcommand{\sgn}{\mathop{\mathrm{sgn}}\nolimits}
\renewcommand{\P}{\mathbb{P}}

\newtheorem{cor}{Corollary}
\newtheorem{thm}{Theorem}
\newtheorem{lem}{Lemma}
\newtheorem{Def}{Definition}
\theoremstyle{definition}
\newtheorem{remark}{Remark}
\usepackage{color}
\definecolor{lilas}{RGB}{182, 102, 210}

\numberwithin{equation}{section}
\begin{document}
\title{On properties of optimal paths in First-passage percolation}
\date{\today}
\author{Shuta Nakajima} 
\address[Shuta Nakajima]
{Research Institute in Mathematical Sciences, 
Kyoto University, Kyoto, Japan}
\email{njima@kurims.kyoto-u.ac.jp}


\begin{abstract}
  
  In this paper, we study some properties of optimal paths in the first passage percolation on $\Z^d$ and show the following: (1) the number of optimal paths has an exponential growth if the distribution has an atom; (2) the means of intersection and union of optimal paths are linear in the distance. For the proofs, we use the resampling argument introduced in [J. van den Berg and H. Kesten. Inequalities for the time constant in first-passage percolation. Ann. Appl. Probab. 56-80, 1993] with suitable adaptions.
\end{abstract}

\maketitle

\section{Introduction}
We consider the first passage percolation (FPP) on the lattice $\Z^d$ with $d\geq{}2$. The model is defined as follows. The vertices are the elements of $\Z^d$. Let us denote by $E^d$  the set of edges:
$$E^d=\{\langle v,w \rangle|~v,w\in\Z^d,~|v-w|_1=1\},$$
where we set $|v-w|_1=\sum^d_{i=1}|v_i-w_i|$ for $v=(v_1,\cdots,v_d)$, $w=(w_1,\cdots,w_d)$. Note that we consider non--oriented edges in this paper, i.e. $\langle v,w \rangle=\langle w,v \rangle$ and we sometimes regard $\langle v,w \rangle$ as a subset of $\Z^d$ with some abuse of notation.

We assign a non-negative random variable $\tau_e$ to each edge $e\in E^d$, called the passage time of $e$. The collection $\tau=\{\tau_e\}_{e\in E^d}$ is assumed to be independent and identically distributed with common distribution function $F$. A path $\Gamma$ is a finite sequence of vertices $(x_1,\cdots,x_l)\subset\Z^d$ such that for any $i\in\{1,\cdots,l-1\}$, $\{x_i,x_{i+1}\}\in E^d$. It is convenient to regard a path as a subset of edges in the following way:
\begin{equation}\label{edge-vertex}
  \Gamma=(x_i)^l_{i=1}=(\{x_{i},x_{i+1}\})^{l-1}_{i=1}.
\end{equation}
  Given a path $\Gamma$, we define the passage time of $\Gamma$ as
$$t(\Gamma)=\sum_{e\in\Gamma}\tau_e.$$
 Given two vertices $v,w\in\Z^d$, we define the {\em first passage time} from $v$ to $w$ as
$$t(v,w)=\inf_{\Gamma:v\to w}t(v,w),$$
 where the infimum is over all finite paths $\Gamma$ starting at $v$ and ending at $w$. A path from $v$ to $w$ is said to be optimal if it attains the first passage time $t(v,w)$, i.e. $t(\Gamma)=t(v,w)$. Denote by $\mathbb{O}(v,w)$ the set of all self--avoiding optimal paths from $v$ to $w$. (Since if the distribution $F$ has an atom at $0$, the number of optimal paths can be infinity, we only consider self--avoiding optimal paths in this paper.) Hereafter, we simply call them optimal paths.

  \begin{Def}[\cite{BK}]
    A distribution $F$ is said to be {\em useful} if $\E[\tau_e]<\infty$ and 
    $$F(F^-)<
    \begin{cases}
    p_c(d) & \text{if $F^-=0$} \\
    \vec{p}_c(d)& \text{otherwise},
    \end{cases}$$
     where $p_c(d)$ and $\vec{p}_c(d)$ stand for the critical probability for $d$--dimensional percolation and oriented percolation model, respectively and $F^-$ is the infimum of the support of $F$.
  \end{Def}
  Note that if $F$ is continuous with a finite first moment, i.e. $\P(\tau_e=a)=0$ for any $a\in\R$ and $\E \tau_e<\infty$, then it is useful. It is well--known that if $F$ is useful, then for any $v,w\in\Z^d$ there exists an optimal path from $v$ to $w$, i.e. $\sharp\mathbb{O}(v,w)\ge 1$, with probability one (see, e.g. \cite{Kes86}. One can see it from Lemma~\ref{useful} therein). It is worth noting that this problem is open for general distribution (see Question 21 in \cite{50}). If $F$ is continuous, then the optimal path is uniquely determined almost surely. Indeed, it is easy to see that for two different paths $\Gamma_2, \Gamma_2$, $\P(t(\Gamma_1)=t(\Gamma_2))=0$. Since the cardinality of finite paths is countable, it follows that 
 $$\P(\forall v,w\in\Z^d,~\sharp\mathbb{O}(v,w)=1)=1.$$
 Contrarily, if $F$ has an atom, then there can be multiple optimal paths. We study the cardinality, the intersection and the union of optimal paths.\\
 
 Let $(\mathbf{e}_i)^d_{i=1}$ be the canonical basis. We only consider the optimal paths from $0$ to $N\mathbf{e}_1$, though all of the results also hold for any direction. For the sake of simplicity, we write $\mathbb{O}_N=\mathbb{O}(0,N\mathbf{e}_1)$.
\subsection {The number of optimal paths}
\begin{thm}\label{thm:number}
  Suppose that $d\geq 2$ and $F$ is useful and there exists $\alpha\in[0,\infty)$ such that $\P(\tau_e=\alpha)>0$. Then, there exists $c>0$ such that
\begin{equation}\label{inf:number}
  \liminf_{N\to\infty}N^{-1}\log{\sharp \mathbb{O}_N}>c\hspace{4mm}a.s.
\end{equation}
\end{thm}
\begin{remark}
  The similar statement of Theorem~\ref{thm:number} was proved for the directed last passage percolation with finite support distributions in \cite{CKNPS} (see section \ref{history} for the details).
\end{remark}
  We can also obtain the corresponding upper bound of \eqref{inf:number} but with a different constant. 
  \begin{thm}\label{sup:number:thm}
    If $\E[\tau_e^2]<\infty$ and $F$ is useful, then there exists $C>0$ such that 
  \begin{equation}\label{sup:number}
      \limsup_{N\to\infty}N^{-1}\log{\sharp \mathbb{O}_N }\leq C\hspace{4mm}a.s.
    \end{equation}
  \end{thm}
It is certainly desirable to have the existence of $\lim_{N\to\infty}N^{-1}\log{\sharp\left[\mathbb{O}_N\right]}$, but it seems to be much harder problem.
\subsection{Intersection and union of optimal paths}
Theorem~\ref{thm:number} tells us that there are exponentially many open paths. Our next results partially reveal how the optimal paths are distributed in the space.
\begin{figure}[b]\label{fig:intersection}
 \includegraphics[width=8.0cm]{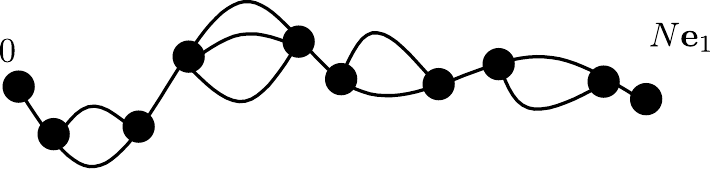}
\caption{}
Schematic picture of intersection of optimal paths.
  \label{fig:two}
\end{figure}

\begin{thm}\label{thm:intersection}
  Suppose that $F$ is useful. Then there exists $c>0$ such that for any $N\in\N$,
  $$\E\left[\sharp \left[\bigcap_{\Gamma\in\mathbb{O}_N}\Gamma\right]\right]\geq{}cN,$$
  where we regard a path $\Gamma$ as a set of edges as in \eqref{edge-vertex}.
\end{thm}

\begin{cor}\label{cor1}
  Suppose that $F$ is useful and $\E\tau_e^2<\infty$. Then there exists $c>0$ such that for any $N\in\N$,
  $$\P\left(\sharp \left[\bigcap_{\Gamma\in\mathbb{O}_N}\Gamma\right]\geq{}cN\right)\geq{}c.$$
\end{cor}
This shows that there are lots of {\em pivotal edges} on optimal paths with positive probability. We believe that the event of the left hand side holds with high probability.
\begin{thm}\label{thm:union}
  Suppose that $F$ is useful and there exists $\ell>2(d-1)$ such that $\E[\tau_e^\ell]<\infty$. Then there exists $C>0$ such that for any $N\in\N$,
  $$\E\left[\sharp \left[\bigcup_{\Gamma\in\mathbb{O}_N}\Gamma\right]\right]\leq{}CN.$$
\end{thm}
\begin{remark}
 A similar statement to Theorem~\ref{thm:union} was proved in \cite{Zhang06} (see section \ref{history2} for the details).
\end{remark}
\begin{remark}
  The moment condition of Theorem~\ref{thm:union} is used in Lemma~\ref{00} and Lemma~\ref{exp}.
\end{remark}
\subsection{Historical background and related works}
First-passage percolation is a model of the spread of a fluid through some random medium, which was introduced by Hammersley and Welsh in \cite{HW65}. Since it is easy to check that $t(\cdot,\cdot)$ is pseudometric and moreover, if $F(0)=0$, exactly a metric almost surely, we can naturally regard the model as a random metric space. Hence mathematical objects of interest are the asymptotic behavior of the first passage time $t(v,w)$ (metric) as $|v-w|_1\to\infty$ and its optimal paths (geodesics). Over 50 years, there has been significant progress for these problems but there still remains many interesting problems (see \cite{50} for more on the background and open problems).
\subsubsection{Number of optimal paths}\label{history}
There has been revived interest on the number of optimal (or maximizing) paths in directed last passage percolation and oriented percolation~\cite{CKNPS,FY12,GGM16}. (Here assigning $\tau_e=0$ if it is open and $\tau_e=\infty$ otherwise, an open path for oriented percolation can be seen to be an optimal path.) Especially, in \cite{CKNPS}, it was proved that the number of maximizing paths of directed last passage percolation with finite support distributions has an exponential growth. The proof was based on the multi-valued map principle (MVMP). In this paper, we prove a similar result by another method, van den Berg-Kesten's resampling argument. Remark that our techniques do not work in critical and supercritical regime, i.e. $\P(\tau_e=0)\ge p_c(d)$ due to the lack of Lemma~\ref{useful}. In supercritical regime, the number of optimal paths should be infinity with high probability.  Even in this case, MVMP seems to be applicable, though not straight forwardly. 

\subsubsection{Intersection and union of optimal paths}\label{history2}
In directed polymer models, the overlap of independent polymers naturally arise in the analysis of the free energy and it has received much 
　interest in this field (see e.g. \cite{Comets-Lec}). In the model which allows us to use Malliavin calculus, especially integration by parts, much is known for the overlaps (see \cite{CC13} and reference therein). However, to my knowledge, little is known for general setting. For the union of optimal paths, it was showed in \cite{Zhang06} that the same estimate of Theorem~\ref{thm:union} for optimal paths from the origin to the boundary of the bounded area if $F$ is subcritical Bernoulli distribution, that is the case $\P(\tau_e\in\{0,1\})=1$ and $\P(\tau_e=0)<p_c(d)$. The proof in \cite{Zhang06} is based on the Russo formula, which seems to be specific to the Bernoulli case.

Note that our results strongly suggest that all optimal paths from $0$ to $N\mathbf{e}_1$ are contained in some thin sausage (see Figure \ref{fig:intersection}). In other words, in practice, these optimal paths should be represented by one optimal path.

\subsection{Notation and terminology}
This subsection collects some notations and terminologies for the proof.
\begin{itemize}
\item We use $c>0$ for a small constant and $C>0$ for a large constant. They may change from line to line.
  \item Given $a\in\R$, let $[a]$ be the greatest integer less than or equal to $a$.
\item Given a path $\gamma=(x_i)^l_{i=1}$, we define a new path as $\gamma[x_m,x_n]=(x_i)^n_{i=m}$.
\item Given two paths $\gamma_1=(x_i)^l_{i=1}$ and $\gamma_2=(y_i)^{l'}_{i=1}$ with $x_l=y_1$, we denote the connected path by $\gamma_1\oplus\gamma_2$, i.e. $\gamma_1\oplus\gamma_2=(x_1\cdots,x_l,y_1,\cdots,y_{l'})$.
\item Given $x,y\in\R^d$, we define $d_{\infty}(x,y)=\max\{|x_i-y_i|~i=1,\cdots,d\}$. It is useful to extend the definition as
  $$d_{\infty}(A,B)=\inf\{d_{\infty}(x,y)|~x\in A,~y\in B\}\text{\hspace{4mm}for $A,B\subset \R^d$}.$$
  When $A=\{x\}$, we write $d_{\infty}(x,B)$.
\item For $x\in\R^d$ and $r>0$, denote the closed ball whose center is $x$ and radius is $\delta$ by $B(x,r)$.
\item Given a set $D\subset\Z^d$, let us define the inner and outer boundary of $D$ as $$\partial^- D=\{v\in D|~\exists w\notin D\text{ such that }|v-w|_1=1\},$$
  $$\partial^+ D=\{v\notin D|~\exists w\in D\text{ such that }|v-w|_1=1\}.$$
  
\item Given a configuration $\tau^{B}=\{\tau^{B}_e\}_{e\in E^d}$ that is a modification of $\tau$, we denote the corresponding first passage time by $t^{B}(v,w)$.
\item Let $F^-$ and $F^+$ be the infimum and supremum of the support of $F$, respectively:
  $$F^-=\inf\{\delta\ge 0|~\P(\tau_e<\delta)>0\},~F^+=\sup\{\delta\ge 0|~\P(\tau_e>\delta)>0\},$$
  where let $F^+=\infty$ if $F$ is unbounded.
  \item We use the convention $a/\infty=0$ for any $a\in\R$.
\end{itemize}

        \section{Proof for Theorem~\ref{thm:number}}
        
Given $e=\langle x,y\rangle$, we sometimes write $\tau_e=\tau(x,y)$ in this section.
\begin{Def} Given a path $\gamma=(x_0,\cdots,x_l)$, let us define the reflection of $x_i$ across $\gamma$ as $x_i^*=x_{i-1}+(x_{i+1}-x_i)$ with the convention $x_0^*=x_0$ and $x_l^*=x_l$. $x\in\Z^d$ is said to be $\mathbb{G}$--turn for $\gamma$ if there exists $i\in\{1,\cdots,l-1\}$ such that $x=x_i$, $\tau(x_{i-1},x_i)+\tau(x_i,x_{i+1})=\tau(x_{i-1},x_i^*)+\tau(x^*_{i},x_{i+1})$, $x_i-x_{i-1}$ is perpendicular to $(x_{i+1}-x_i)$ and  $x_i^*\notin\gamma$.
\end{Def}
\begin{Def}
  We set the attached passage time as $t^+(\gamma)=t(\gamma)+\beta\sharp\{x_i|~x_i\text{ is $\mathbb{G}$--turn for $\gamma$}\}$ with a constant $\beta>0$ to be chosen in Lemma~\ref{ineq} and denote by $t^+(0,N\mathbf{e}_1)$ the first passage time from $0$ to $N\mathbf{e}_1$ corresponding to $t^+(\cdot)$. We call it the attached first passage time.
\end{Def}
\begin{lem}\label{lem:exp2}
  \begin{equation}\label{exp2}
  \E[\min\{\sharp{}\{x\in\Gamma|~\text{$x$ is $\mathbb{G}$--turn for $\Gamma$}\}|~\Gamma\in \mathbb{O}^+_N\}]\geq{}cN.
  \end{equation}
   Note that minimum of \eqref{exp2} is over all optimal path for $t^+(\gamma)$.
\end{lem}
We postpone the proof and first prove Theorem~\ref{thm:number}. By the definition of $\mathbb{G}$--turn, we have
\begin{equation}\label{basic}
  \begin{split}
    \beta\min\{\sharp{}\{x\in\Gamma|~\text{$x$ is $\mathbb{G}$--turn for $\Gamma$}\}|~\Gamma\in \mathbb{O}^+_N\}&\leq t^+(0,N\mathbf{e}_1)-t(0,N\mathbf{e}_1)\\
    &\leq \beta\min\{\sharp{}\{x\in\Gamma|~\text{$x$ is $\mathbb{G}$--turn for $\Gamma$}\}|~\Gamma\in \mathbb{O}_N\}.
    \end{split}
\end{equation}
Indeed, taking $\Gamma\in \mathbb{O}^+_N$ which attains the minimum, by the definition of $t^+(\cdot)$, we have $$t^+(0,N\mathbf{e}_1)-t(0,N\mathbf{e}_1)\ge t^+(\Gamma)-t(\Gamma)= \beta\min\{\sharp{}\{x\in\Gamma|~\text{$x$ is $\mathbb{G}$--turn for $\Gamma$}\}|~\Gamma\in \mathbb{O}^+_N\},$$
which implies the first inequality. For the second inequality of \eqref{basic}, we only take $\Gamma\in\mathbb{O}_N$ attaining the minimum and calculate $t^+(\Gamma)-t(\Gamma)\ge  t^+(0,N\mathbf{e}_1)-t(0,N\mathbf{e}_1)$.\\

By Kingman's subadditive ergodic theorem, there exist $\mu,\mu^+\ge 0$ such that almost surely,
$$\lim_{N\to\infty} N^{-1}(t^+(0,N\mathbf{e}_1)-t(0,N\mathbf{e}_1))= \lim_{N\to\infty} N^{-1}(\E t^+(0,N\mathbf{e}_1)-\E t(0,N\mathbf{e}_1))=\mu^+-\mu.$$
On the other hand, by Lemma~\ref{lem:exp2} and \eqref{basic}, we have $\mu^+-\mu\ge c$. This together with \eqref{basic} leads to

\begin{equation}\label{basic2}
  \begin{split}
    \min\{\sharp{}\{x\in\Gamma|~\text{$x$ is $\mathbb{G}$--turn for $\Gamma$}\}|~\Gamma\in \mathbb{O}_N\}\geq c N.
    \end{split}
\end{equation}
Let us take an arbitrary optimal $\Gamma=(x_i)^l_{i=1}$ path satisfying $\sharp{}\{x\in\Gamma|~\text{$x$ is $\mathbb{G}$--turn for $\Gamma$}\}\geq c N.$
Let us define $x_i^{\G}$ as $x_i^{\G}=x_i^*$ if $x_i$ is $\G$--turn and $x_i^{\G}=x_i$ otherwise. Note that for any choice $y_i\in \{x_i,x_i^{\G}\}$, $(y_i)^l_{i=1}$ is optimal path, i.e. $t((y_i))=t(0,N\mathbf{e}_1)$. Although it may be not self--avoiding, since the number of overlaps at any vertex $x$, i.e. $\sharp\{y_i|~y_i^*=x\}$, is at most $2d$, \eqref{basic2} yields $\sharp \mathbb{O}_N \geq 2^{cN/2d}$ as desired.
\begin{proof}[Proof of Lemma~\ref{lem:exp2}]
 Recall that $F^+$ be the suprimum of the support of $F$.  
We take $n>0$ sufficiently large depending on the distribution $F$ but not depending on $N$. We prepare three kinds of box whose notations are the same as in \cite{BK} (See Figure \ref{fig:two}). First, define the hypercubes $S(l;n)$, for $l\in\Z^d$ and $n\in\N$, by 
$$S(l;n)=\{v\in\Z^d:nl\le v_i< n(l+1)\text{ for any $i$}\}.$$
We call these hypercubes $n$-cubes. Second, we difine the large $n-$cubes $T(l;n)$, for $l\in\Z^d$, by $$T(l;n)=\{v\in\Z^d:nl-n\le v_i\le n(l+2)\text{ for any $i$}\}.$$
Finally, we difine the $n$-boxes $B^j(l;n)$, for $l\in\Z^d$ and $j\in\{1,\cdots,d\}$, as
$$B^j(l;n)=T(l;n)\cap{}T(l+2\sgn(j)\mathbf{e}_{|j|};n).$$
Note that $S(l;n)\subset{}T(l;n)$ and $B^j(l;n)$ is a box of size $3n\times\cdots\times{}3n\times{}n\times{}3n\cdots\times{}3n$, where $n$ is the length in $i$-th coordinate. For the simplicity of notation, we set $B=B^j(l;n)$. We take sufficiently large $M>0$ to be chosen later.\\

 The following is the crucial property of useful distribution.
\begin{lem}\label{useful}
If $F$ is useful, then there exsits $\delta_1>0$ and $D>0$ such that for any $v,w\in\Z^d$,
$$\P(t(v,w)<(F^-+\delta_1)|v-w|_1)\leq{}e^{-D|v-w|_1}.$$
\end{lem}
For the proof of this lemma, see Lemma~5.5 in \cite{BK}. We take $\delta_1>0$ as in Lemma~\ref{useful}.
 \begin{Def} Let us consider the following conditions:\\
   
\noindent (1)for any $v,w\in B^j(l;n)$ with $|v-w|_1\geq{}n^{1/3}$,
$$t(v,w)\geq{}(F^{-}+\delta_1)|v-w|_1,$$
where $\delta_1>0$ is as in Lemma~\ref{useful}.\\

\noindent(2)for any $e\cap B\neq\emptyset$, $\tau_e\leq{}M$.\\

\noindent(3)for any $e\cap B\neq\emptyset$, $\tau_e\leq{}F^+-M^{-1}$.\\

\noindent{}An $n$-box $B$ is said to be black if
$\begin{cases}
 \text{$F^+=\infty$ and }\text{(1) and (2) hold,}\\
 \text{$F^+<\infty$, $\P(\tau_e=F^+)=0$ and }\text{(1) and (3) hold, or}\\
 \text{$F^+<\infty$, $\P(\tau_e=F^+)>0$ and }\text{(1) holds}.\\
\end{cases}$
 \end{Def}
 \begin{figure}[b]
  \includegraphics[width=4.0cm]{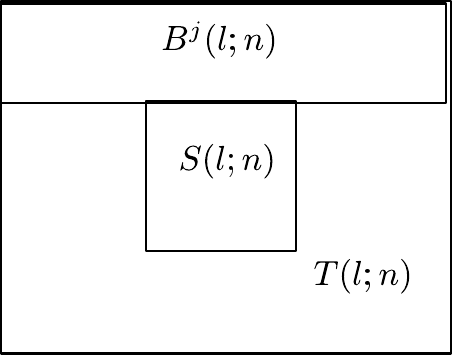}
  \hspace{8mm}
  \includegraphics[width=4.0cm]{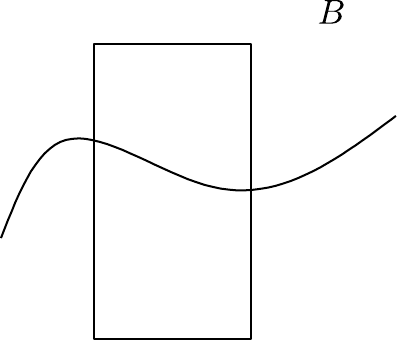}
\caption{}
Left: Boxes: $S$, $T$, $B$.\\
Right: $\mathbb{O}_N$ crosses a $n$-box in the short direction.\\
  \label{fig:two}
\end{figure}
 It is easy to check that if $M=M(n)$ is suffiently large, then $\P(B\text{ is black})\to{}1$ as $n\to\infty$. Indeed the first condition which appears in definition of blackness holds with high probability by using Lemma~\ref{useful} (See also (5.5), (5.31) and (5.32) in \cite{BK}). Together with a similar argument (Peierls argument) of (5.2) in \cite{BK}, the following lemma follows. 
 \begin{lem}\label{lem:p}
There exist $\epsilon,D,n_1,M_1>0$ such that for any $N\in\N$, $n\geq n_1$ and $M\geq M_1$,\\
 $$\text{$\P(\exists$ path from $0\to{}N\mathbf{e}_1$ which visits at most $\epsilon{}N$ distinct black $n$-cubes$)\leq{}e^{-DN}$}$$
 \end{lem}
 
A path which starts in $S(l;n)$ and ends outside of $T(l;N)$ must have a segment which lies entirely in one of the surrounding $n$-boxes, and which connects the two opposite large faces of that $n$-box. This means that this path crosses at least one $n$-box in the short direction (See Figure \ref{fig:two}). Hereafter ``crossing an $n$-box" means crossing in the short direction.
\begin{Def}
  ~\\
\noindent{}An $n$-box $B$ is said to be white if
there exists $\Gamma\in\mathbb{O}_N$ such that $\Gamma$ cross $B$. \\

\noindent{}An $n$-box $B$ is said to be gray if $B$ is black and white.

 \end{Def}

From these observations together with Lemma~\ref{lem:p}, we obtain
\begin{equation}\label{numb}
  \E[\sharp\{\text{distinct gray $n$-box $B$} \}]\geq{}\epsilon{}N/2
  \end{equation}

\begin{Def}
  Define\\
  
  \noindent$$F^+_M=\begin{cases}
  \text{$M$}&\text{ if $F^+=\infty$,}\\
  \text{$F^+-M^{-2}$}&\text{ if $F^+<\infty$ and $F(\{F^+\})=0$,}\\
  \text{$F^+$}&\text{ if $F^+<\infty$ and $F(\{F^+)>0$,}\\
  \end{cases}$$
  and\\
  $$F^-_M=\begin{cases}
  \text{$F^-+M^{-2}$}&\text{ if $F(\{F^-\})=0$,}\\
  \text{$F^-$}&\text{ if $F(\{F^-\})>0$.}
\end{cases}$$
\end{Def}
Note that if $M$ is sufficiently large,
\begin{equation}\label{estimate}
  F^-_M<F^-+\delta_1/2< F^+_M\text{ and }F^-_M\leq \alpha\leq F^+_M,
\end{equation}
where $\alpha$ first appears in Theorem~\ref{thm:number}.
\begin{Def}
  ~\\
 A point $x\in\Z^d$ is said to be a turn for $\gamma=(x_0,\cdots,x_{|\gamma|})$ if there exists $i\in\{1,\cdots,l-1\}$ such that $x=x_i$, $x_i-x_{i-1}$ is perpendicular to $x_{i+1}-x_i$. Otherwise, we say that $x_i$ is flat.\\
  \noindent An $n$-box $B$ is called $\mathbb{G}$--turn if for any $\Gamma\in\mathbb{O}^+_N$, there exists $x\in B$ such that $x$ is $\mathbb{G}$--turn for $\Gamma$.
 \end{Def}

     Denote by $\partial^+ B$ the outer boundary of $B$. The following lemma will be used in Definition \ref{gamab}.
  \begin{lem}\label{conc:gam}
    Suppose that $F^+<\infty$. If we take $n$ sufficiently large, then for any $a,b\in\partial^+ B$ with \\
    $|a-b|_1\geq{}\delta_1n/(2F^+)=:C(\delta_1,F^+)n$, there exists a self-avoiding path $\gamma_{a,b}=(x_0,\cdots,x_{|\gamma_{a,b}|})$ from $a$ to $b$ satisfying $\{x_i\}^{|\gamma_{a,b}|-1}_{i=1}\subset B$ such that the following hold:
   \begin{equation}\label{gam}
     \begin{array}{l}
       \text{ (1) $|x_i-x_j|_1=|i-j|$ if $|i-j|\leq 12dn^{1/3}$},\\
       \text{ (2) $|x_i-x_j|_1\geq|i-j|-C\sqrt{|i-j|}$ for any $i,j$, where $C:={}800d\left(1+C(\delta_1,F^+)^{-1/2}\right)$},\\
       \text{ (3) if  $0\leq i<j\leq{}|\gamma_{a,b}|$ and $|i-j|\geq{}3\sqrt{n}$, $\exists q\in\{i+1,\cdots,j-1\}$ s.t. $x_q$ is turn for $\gamma_{a,b}$},\\
       \text{ (4) if $x_p$ and $x_q$ are distinct turns for $\gamma_{a,b}$, $|x_p-x_q|_1>{}4$},\\
       \text{ (5) $|\gamma_{a,b}|\leq |a-b|_1+100d\sqrt{n} $},\\
       \text{ (6) $d_{\infty}(x_i,B^c)\geq{}4dn^{1/3}$ for $2dn^{1/2}\leq i\leq{}|\gamma_{a,b}|-2dn^{1/2}$},\\
       \text{ (7) for any $i<j$ with $|i-j|\leq{}\sqrt{n/2}$, $\sharp \{i\leq p \leq j|~\text{$x_p$ is turn for $\gamma_{a,b}$}\}\leq 2$},\\
       \end{array}
     \end{equation}    
  \end{lem}
  The reason why we use $\sqrt{n}$ and $n^{1/3}$ above is just $n^{1/3}\ll\sqrt{n}\ll n$ and not important.
  \begin{proof}
    First we take a self-avoiding path $\tilde{\gamma}=(\tilde{x}_0,\cdots,\tilde{x}_l)$ from $a$ to $b$ so that
    \begin{enumerate}
    \item $\forall i,j$ with $|i-j|\leq |a-b|_1/4$, $|\tilde{x}_i-\tilde{x}_j|_1=|i-j|$,
      \item If neither $\tilde{x}_i$ nor $\tilde{x}_j$ is flat for $\tilde{\gamma}$, $|\tilde{x}_i-\tilde{x}_j|_1\geq \sqrt{n}$,
    \item $|\tilde{\gamma}|\leq |a-b|_1+10d\sqrt{n}$,
    \item If $2dn^{1/3}\leq i \leq |\tilde{\gamma}|-2dn^{1/3}$, then $d_{\infty}(\tilde{x}_i,B^c)\geq 8dn^{1/3}$.
      \end{enumerate}
    Then we consider rough boxes as in Figure \ref{fig:three}. For the construction of $\gamma_{a,b}$, we attach rough box in the middle of the $3\sqrt{n}$ successive flat points of $\tilde{\gamma}$ if they exist and we continue the procedure until they vanish. We take $\gamma_{a,b}=(x_0,\cdots,x_{|\gamma_{a,b}})$ which was obtained by the above procedure eventually. Note that the number of rough boxes attached in this procedure is at most $|\tilde{\gamma}|/\sqrt{n}$ $(\leq 6d\sqrt{n}).$ Thus, $|\gamma_{a,b}|\leq 5\times 2\times 6d\sqrt{n}+|\tilde{\gamma}|\leq |a-b|_1+100d\sqrt{n}$ which implies (5). (1), (3), (4), (6), (7) are trivial by the way of construction. Finally, we will prove (2). If $|i-j|\leq |a-b|_1/8$, then
    $$|x_i-x_j|\geq |i-j|-(5\times 2\times |i-j|/\sqrt{n})\geq |i-j|-20d\sqrt{|i-j|}.$$
    Otherwise, since $|i-j|-|x_i-x_j|_1\leq |\gamma_{a,b}|-|a-b|_1\leq 100d\sqrt{n}$ and $|a-b|_1\geq{}C(\delta_1,F^+)n$,
    $$|x_i-x_j|\geq |i-j|-100d\sqrt{n}\geq |i-j|-\frac{800d}{\sqrt{C(\delta_1,F^+)}}\sqrt{|i-j|}.$$
  \end{proof}
  For general $a,b$, we get the following lemma.
  \begin{lem}\label{conc:gam2}
    For any $a,b\in\partial^+ B$ with $a\neq b$, there exists a self-avoiding path $\gamma_{a,b}=(x_0,\cdots,x_{|\gamma_{a,b}|})$ from $a$ to $b$ satisfying $\{x_i\}^{|\gamma_{a,b}|-1}_{i=1}\subset B$ such that the following hold:
     \begin{enumerate}
     \item there exists at least one turn $x_i$ for $\gamma_{a,b}$ such that $d_{\infty}(x_i,B^c)\geq{}2$,
       \item $|a-b|_1+4d\sqrt{n}\geq |\gamma_{a,b}|$.
       \end{enumerate}
  \end{lem}
  The proof is the same (or much simpler) as in Lemma~\ref{conc:gam}, so we skip the details.
  \begin{figure}[h]
  
   \includegraphics[width=4.0cm]{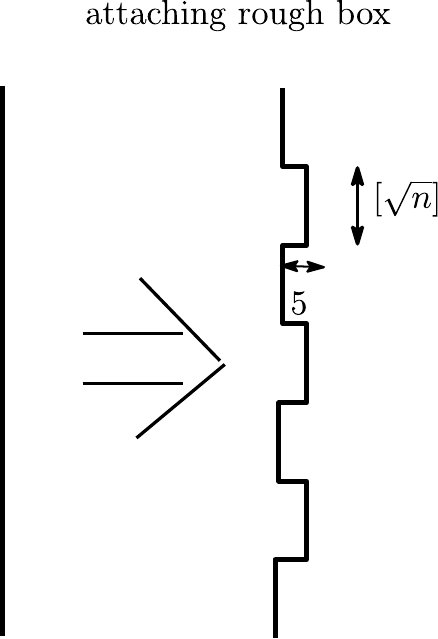}
  \hspace{10mm}
  \includegraphics[width=4.0cm]{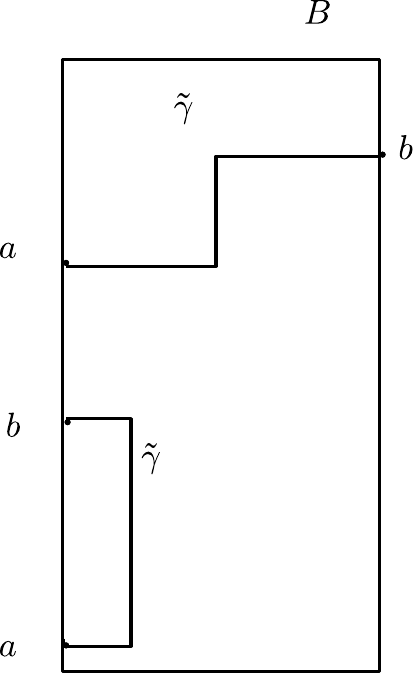}
  \hspace{10mm}
  \includegraphics[width=4.0cm]{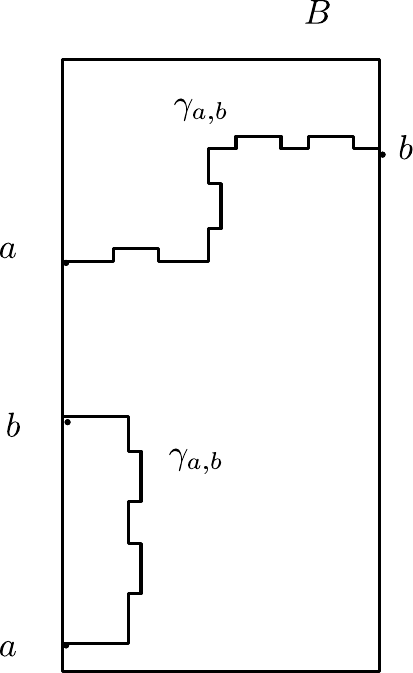}
\caption{}
Left: rough boxes. Middle: Examples of $\tilde{\gamma}$. Right: Examples of $\gamma_{a,b}$.
  \label{fig:three}
\end{figure}
 \begin{Def}\label{gamab}
   For any $a,b\in\partial^+ B$ with $|a-b|_1\geq{}\delta_1n/(2F^+)+1$, we take a self avoiding path $\gamma_{a,b}=(x_0,\cdots,x_{|\gamma_{a,b}|})$ from $a$ to $b$ with $\{x_i\}^{|\gamma_{a,b}|-1}_{i=1}\subset B$ so that if $F^+=\infty$, then Lemma~\ref{conc:gam2} holds and otherwise, if $F^+=\infty$, then Lemma~\ref{conc:gam} holds.\\
   
   \noindent If $|a-b|_1<\frac{\delta_1n}{2F^+}+1$, then we take an arbitrary selfavoiding path $\gamma_{a,b}\subset B\cap\partial^+ B$ from $a$ to $b$.
   
 \end{Def}
 ~\\
 
 Given a  path $\gamma=\gamma_{a,b}=(x_0,\cdots,x_{|\gamma|})$ and $n$-box $B$, $\tau$ is said to be satisfied $(\gamma,B)$-condition if (1) $\tau(x_{i-1},x_i)=\tau(x_{i},x_{i+1})=\tau(x_{i-1},x_{i}^*)=\tau(x_{i}^*,x_{i+1})=\alpha$ if $x_i$ is turn for $\gamma$, (2) $\tau_e\leq F^-_M$ for $e\in\gamma$ with $e\notin \{\{x_{i-1},x_i\}|~\text{$x_i$ is turn for $\gamma$}\}\cup\{\{x_i,x_{i+1}\}|~\text{$x_i$ is turn for $\gamma$}\}$, (3) $\tau_e\geq{}F^+_M$ for other edges with $e\cap B\neq\emptyset$. Denote the independent copy of $\tau$ by $\tau^*$ and set $\tau^{B}$ as $\tau^{B}_e=\tau_e^{*}$ if $e\cap{}B\neq\emptyset$, $\tau^{B}_e=\tau_e$ otherwise. Let $(\tilde{a},\tilde{b})$ be random variable on $\partial^+ B\times \partial^+ B$ with uniform distribution and its probability measure $P$. Given a path $\Gamma=(x_0,\cdots,x_l)$ and a $n$-box $B$, we set
 $$\text{st}(\Gamma,B)=x_{\min\{i|~x_i\in \partial^+ B\}}\text{, fin}(\Gamma,B)=x_{\max\{i|~x_i\in \partial^+ B\}}.$$
 Note that if $\Gamma$ cross $B$ and $B$ is black, since $\text{$t^+($st$(\Gamma,B)$,fin$(\Gamma,B))$}\geq{}(F^-+\delta_1)n$, $$\text{$|$st$(\Gamma,B)$-fin$(\Gamma,B)$}|_1\geq \frac{(F^-+\delta_1)n}{2F^+}+1.$$
 Here the above inequality holds even if $F^+=\infty$.
 \begin{lem}\label{ineq}
 We take $\beta=M^{-2}$. If $M\geq{}n^{2d}$ and $n$ is sufficiently large, there exists $c>0$ such that for any $N\in\N$, unless $0\in B$ or $N\mathbf{e}_1\in B$,
 \begin{equation}\label{crucial2}
\begin{split}
   &\P(\text{ $B$ is $\mathbb{G}$--turn for $\tau$})=P\otimes\P(\text{ $B$ is $\mathbb{G}$--turn for $\tau^{B}$})\\
  & \geq{}P\otimes\P\left(
    \begin{array}{c}
      \text{$B$ is gray for $\tau$, $\exists\Gamma\in\mathbb{O}^+_N$ s.t. $\Gamma$ cross $B$,} \\
      \text{$(\tilde{a},\tilde{b})=(\text{\rm{st}}(\Gamma,B),\text{\rm{fin}}(\Gamma,B)))$ and $\tau^*$ satisfies $(\gamma_{\tilde{a},\tilde{b}}$,$B)$-condition}\\
\end{array}
  \right)\\
  &=\frac{1}{|\partial^+{}B|^2}\sum_{(a,b)}\P\left(
  \begin{array}{c}
    \text{$B$ is gray for $\tau$, $\exists\Gamma\in\mathbb{O}^+_N$ s.t. $\Gamma$ cross $B$,}\\
    \text{$(a,b)=(\text{\rm{st}}(\Gamma,B),\text{\rm{fin}}(\Gamma,B)))$, $\tau^*$ satisfies $(\gamma_{a,b}$,$B)$-cond.}\\
    \end{array}\right)\\
 &\geq{}c\P(\text{$B$ is gray}).\\
\end{split}
 \end{equation}
 \end{lem}
 \begin{proof}
   We first prove the first inequality. We suppose that $B$ is gray for $\tau$, there exists $\Gamma\in\mathbb{O}^+_N$ such that $(\tilde{a},\tilde{b})=(\text{st}(\Gamma,B),\text{fin}(\Gamma,B)))$ and $\tau^*$ satisfy $(\gamma_{\tilde{a},\tilde{b}}$,$B)$--condition.  We write $\gamma=\gamma_{\tilde{a},\tilde{b}}=(x_0,\cdots,x_{|\gamma|})$, $a=\tilde{a}$ and $b=\tilde{b}$ for simplicity. Let us define $\gamma^*$ as $\gamma^*=(x_i^*)^{|\gamma|}_{i=0}$. Our first goal is to prove that $B$ is $\mathbb{G}$--turn for $\tau^{B}$ under this condition. To this end, we take $\Gamma^B$ to be an optimal path for the attached first passage time with respect to $\tau^B$. \\

   First we consider the case $F^+=\infty$. Let us denote by $t^{B,+}$ the attached first passage time with respect to $\tau_e^B$. The $(\gamma$,$B)$--condition and blackness of $B$ lead to
 \begin{equation}\label{inf1}
   \begin{split}
     t^{B,+}(0,N\mathbf{e}_1)&\leq t^{B,+}(\Gamma[0,a]\oplus \gamma \oplus\Gamma[b,N\mathbf{e}_1])\\
     &\leq t^{B,+}(\Gamma[0,a])+ (F^-+\delta_1/2)|a-b|_1 + t^{B,+}(\Gamma[b,N\mathbf{e}_1])+2d\beta |B|\\
     &< t^{+}(\Gamma[0,a])+ (F^-+\delta_1)|a-b|_1 + t^{+}(\Gamma[b,N\mathbf{e}_1])\\
     &\leq t^+(\Gamma)=t^+(0,N\mathbf{e}_1).
  \end{split}
 \end{equation}
 Therefore, $\Gamma^B$ must enter $B$, i.e. there exists $e\in \Gamma^B$ such that $e\cap B\neq\emptyset$. We will show that
 \begin{equation}\label{huhho}
   \Gamma^B\cap \{e\in E^d|~e\cap B\neq\emptyset,~e\not\subset (\gamma\cup \gamma^*)\}=\emptyset.
 \end{equation}
 In fact, if such an edge exists, by $(\gamma$,$B)$--condition, then $$\sum_{\underset{e\cap B\neq\emptyset}{e\in\Gamma^B}}\tau_e^B\geq M^2.$$ On the other hand, since $\sum_{\underset{e\cap B\neq\emptyset}{e\in\Gamma^B}}\tau_e\leq M^2/2$, we obtain
 \begin{equation}\label{inf2}
   \begin{split}
     t^{+,B}(0,N\mathbf{e}_1)&=t^{+,B}(\Gamma^B)\\
     &\geq t^{+,B}(\Gamma^B[0,\text{st}(\Gamma^B,B)])+M^2 + t^{+,B}(\Gamma[\text{fin}(\Gamma^B,B),N\mathbf{e}_1])\\
     &\geq t^{+}(\Gamma^B[0,\text{st}(\Gamma^B,B)])+M^2/2 + t^{+}(\Gamma[\text{fin}(\Gamma^B,B),N\mathbf{e}_1])+2d\beta|B| \\
     &\geq t^+(\Gamma^B)\geq t^+(0,N\mathbf{e}_1),
  \end{split}
 \end{equation}
 which contradicts \eqref{inf1}.\\
 
 Therefore, since $\Gamma^B$ is self--avoiding, \eqref{huhho} implies that $a,b\in \Gamma^B$ and there exists a path $\gamma':a\to b$ with $\gamma'\subset \gamma\cup \gamma^*$ such that $(z_0,\cdots,z_{|\gamma|})={}\Gamma^B[a,b]$. Since $\Gamma^B$ is an arbitrary optimal path, it yields that $B$ is $\mathbb{G}$--turn for $\tau^B$.  \\

 Hereafter we suppose that $F^+<\infty$. Then since $\Gamma$ cross $B$ and $B$ is black, we have $|a-b|_1\geq{}(F^-+\delta_1)n/(2F^++1)$ and $t^+(a,b)\geq{}(F^-+\delta_1)|a-b|_1$. By (5) and (7) of \eqref{gam}, we obtain
$$\sharp\{x\in\gamma|~\text{$x$ is turn for $\gamma$}\}\leq 16d\sqrt{n}\text{ and }\sharp\{e\in E^d|~e\in \gamma\}\leq |a-b|_1+100d\sqrt{n} .$$
 It follows from the $(\gamma$,$B)$--condition that
  \begin{equation}
    \begin{split}
      t^B(\gamma)&\leq F^-_M(|a-b|_1+100d\sqrt{n})+2\alpha\sharp\{x\in\gamma|~\text{$x$ is turn for $\gamma$}\} \\
      &\leq F^-_M(|a-b|_1+100d\sqrt{n})+ 32\alpha d\sqrt{n}.
       \end{split}
 \end{equation}
 This yields that 

 \begin{equation}\label{finish}
   \begin{split}
     t^+(\Gamma^B)-t^{B,+}(\Gamma^B)&\geq t^+(0,N\mathbf{e}_1)-t^{B,+}(0,N\mathbf{e}_1)\\
     &\geq{}t^+(\Gamma)-t^{B,+}(\Gamma[0,a]\oplus \gamma \oplus\Gamma[b,N\mathbf{e}_1])\\
     &\geq{}t^+(a,b)-t^B(\gamma)-\beta(\sharp\{x\in\gamma|~\text{$x$ is turn for $\gamma$}\}+2)\geq{}\delta_1 n/4.
     \end{split}
 \end{equation}
 On the other hand, since $\tau_e\leq{}\tau^B_e$ unless $e\in\gamma\cup\gamma^*$,
 \begin{equation}\label{finish2}
   \begin{split}
     t^+(\Gamma^B)-t^{B,+}(\Gamma^B)\leq 2F^+\cdot 16d\sqrt{n}+F^+\sharp\{e\in E^d|~e\in (\gamma\cup\gamma^*) \cap \Gamma^B\}+2d\beta|B|.
     \end{split}
 \end{equation}

 Comparing \eqref{finish} with \eqref{finish2}, taking $n$ sufficiently large, we have
 $$\sharp\{e\in E^d|~e\in (\gamma\cup\gamma^*) \cap \Gamma^B\}\geq{}\delta_1 n/(4F^+)-32d\sqrt{n}-2d\beta|B|>\delta_1 n/(8F^+).$$ In particular, there exist $p,q \in \N\cup\{0\}$ such that $8d\sqrt{n}<p<q<|\gamma|-8d\sqrt{n}$, $|p-q|\geq{}8d\sqrt{n}$ and $x_p,x_q\in \Gamma^B$. To end the proof of the first inequality, we will use the following lemma to control $\Gamma^B$ not detouring from $\gamma$ often.
 \begin{lem}\label{lem:detour}
   Consider $4d\sqrt{n}\leq p_1\leq{}|\gamma|-4d\sqrt{n}$ and $0\leq q_1\leq |\gamma|$. Then under the above condition, for any self avoiding path $(y_0,\cdots,y_l)$ which satisfies $y_0\in\{x_{p_1},x_{p_1}^*\}$, $y_l\in\{x_{q_1},x_{q_1}^*\}$, $y_i\notin(\gamma\cup \gamma^*)$ for $1\leq i\leq l-1$,
   \begin{equation}\label{bonbon}
     t^{B,+}(y_0,\cdots,y_l)>{}t^{B,+}(\{y_0,x_{p_1+1}\}\oplus\gamma[x_{p_1+1},x_{q_1-1}]\oplus\{x_{q_1-1},y_l\})+2\beta.
   \end{equation}
    \end{lem}
 Note that $2\beta$ in \eqref{bonbon} is necessary because $x_{p_1}$ and $x_{q_1}$ may be $\mathbb{G}$--turn for $\Gamma^B$

 \begin{proof}
   Without loss of generality, we suppose that $p_1< q_1$. If $|p_1-q_1|>4dn^{1/3}$, by blackness of $B$ and (2) and (7) of \eqref{gam}, we have
\begin{equation}
\begin{split}
  t^{B,+}(y_0,x_{p_1+1},\cdots,x_{q_1-1},y_l)+2\beta&<2(2\alpha+\beta)\{(q_1-p_1)(n/2)^{-1/2}+1\}+F^-_M(q_1-p_1)\\
  &<(F^-+\delta_1)|x_{p_1}-x_{q_1}|_1\leq t^{B,+}(y_0,\cdots,y_l).
    \end{split}
\end{equation}
Next we assume $|p_1-q_1|\leq{}4dn^{1/3}$. Note that $l\ge 2$. If $l=2$ and $x_{p_1+1}$ is turn for $\gamma$, $y_1$ must be an element of $\{x_{p_1+1},x_{p_1+1}^*\}$, which is a contradiction. On the other hand, if $l=2$ and $x_{p_1+1}$ is not turn for $\gamma$, we have
$$t^{B,+}(y_0,x_{p_1+1},y_l)+2\beta\le \alpha+F^-_M+2\beta<2F^+_M\le t^{B,+}(y_0,\cdots,y_l),$$
as desired. We suppose that $l\ge 3$. Note that by (4), (6), (7) of \eqref{gam},
   $$(4dn^{1/3}\land l)F^+_M\leq{}t^{B,+}(y_0,\cdots,y_l),$$
   $$t^{B,+}(y_0,x_{p_1+1},\cdots,x_{q_1-1},y_l)\leq 2\alpha+\beta(q_1-p_1)+F^-_M(q_1-p_1-2)+2(\alpha-F^-_M)\mathbb{I}_{\{l>4\}}.$$
  Since $4dn^{1/3}\land l\geq{}|x_{q_1}-x_{p_1}|_1=q_1-p_1$ by (1) of \eqref{gam}, by using \eqref{estimate}, we obtain $${}t^{B,+}(y_0,x_{p_1+1},\cdots,x_{q_1-1},y_l)+2\beta<t^{B,+}(y_0,\cdots,y_l).$$
 \end{proof}
 By Lemma~\ref{lem:detour}, since $\Gamma^B$ is optimal, there exist $4d\sqrt{n}\le p_1<q_1\le |\gamma|-4d\sqrt{n}$ and a path $\gamma'=(y_i)^l_{i=0}$ with $y_0\in\{x_{p_1},x_{p_1}^*\}$, $y_l\in\{x_{q_1},x_{q_1}^*\}$ such that $y_{0},y_l\in\Gamma^B$, $|p_1-q_1|\ge 4d\sqrt{n}$, $\gamma'\subset \gamma[x_{p_1},x_{q_1}]\cup \{x_i^*|~p_1\le i\le q_1\}$ and $\gamma'=\Gamma^B[y_0,y_l]$. Then we have by (3) of \eqref{gam} that there exists $i\in \{p_1+[\sqrt{n}],\cdots,q_1-[\sqrt{n}]\}$ such that $x_i$ is $\mathbb{G}$--turn for $\gamma$ and $(x_{i-1},x_i),(x_{i},x_{i+1})\in\Gamma^B$ or $(x_{i-1},x_i^*),(x_{i}^*,x_{i+1})\in\Gamma^B$. Without loss of generality, we suppose that $(x_{i-1},x_i)\in\Gamma^B$ and $(x_{i},x_{i+1})\in\Gamma^B$. Since $\Gamma^B$ is optimal and $\gamma'=\Gamma^B[y_0,y_l]$, by the $(\gamma,B)$--condition, $\Gamma^B$ never come back to $x_i^*$, which implies $x^*_i\notin \Gamma^B$. In particular $x_i$ is $\mathbb{G}$--turn for $\Gamma^B$. Since $\Gamma^B$ is arbitrary, the first inequality follows.

 Second inequality of \eqref{crucial2} follows from the fact that there exists $c>0$ such that for any $a,b\in\partial^+{}B$, $$\P(\text{$\tau^*$ satisfy $(\gamma$,$B)$-condition})\geq{}c.$$
 \end{proof}
 From \eqref{numb} and \eqref{crucial2}, we have
\begin{equation}
\begin{split}
  \E\left[\min\{\sharp{}\{x\in\Gamma|~\text{$x$ is $\mathbb{G}$--turn for $\Gamma$}\}|~\Gamma\in \mathbb{O}^+_N\}\right]&\geq{}\frac{1}{2d}\sum_{B^j(l;n):n\text{-box}}\P(B^j(l;n)\text{ is $\mathbb{G}$--turn})\\
  &\geq \frac{c}{2d}\sum_{B^j(l;n):n\text{-box}}\P(B^j(l;n)\text{ is gray})\\
&\geq \frac{c}{2d}\E[\sharp\{\text{distinct gray $n$-box $B$} \}]\geq{}cN,
  \end{split}
\end{equation}
where $2d$ appears because of the overlap of $n$-boxes. Thus the proof is completed.
\end{proof}
As a result, we have the following corollary.
\begin{cor}
  There exists $c>0$ such that 
  $$\P(\text{ for any optimal path $\Gamma\in\mathbb{O}_N$, the number of $\mathbb{G}$--turn points of $\Gamma$ is at least $cN$})\to 1.$$
\end{cor}
\begin{proof}[Proof of Theorem~\ref{sup:number:thm}]
  we first define the event as
$$A=\{\forall \Gamma\in\mathbb{O}_N\text{ such that }\sharp \Gamma\leq KN\}.$$
Then under the event $A$, we have $\sharp \mathbb{O}_N\leq (2d)^{KN}$, which yields that
$$\frac{1}{N}\log{\sharp \mathbb{O}_N}\le K\log{2d}.$$
On the other hand, by \eqref{fell} below, there exists $K>0$ such that for any $N>0$
$$\mathbb{P}(\exists \Gamma\in\mathbb{O}_N\text{ such that }\sharp \Gamma>  KN)\leq K N^{-2d}.$$
By the Borel--Cantelli lemma, we complete the proof of \eqref{sup:number}.
\end{proof}
\section{Proof of Theorem~\ref{thm:intersection}.}
We begin with the connection between the intersection and restricted union of optimal paths.
\begin{lem}\label{connection}
  For any $\alpha>F^-$, there exists $c>0$ such that for any $N\in\N$,
  $$\E\left[\sharp\left[\bigcap_{\Gamma\in\mathbb{O}_N}\Gamma\right]\right]\ge c\E\left[\sharp \left\{\eta\in E(\Z^d)|~\exists \Gamma \in \mathbb{O}_N,~\eta\in\Gamma,~\tau_{\eta}>\alpha\right\}\right].$$
\end{lem}
\begin{proof}
Let $\tau^*$ be an independent copy of $\tau$. Given an edge $\eta\in E(\Z^d)$, we set $\tau^{(\eta)}$ as
   $$\tau^{(\eta)}_{e}=  \begin{cases}
    \tau_{e}^* & \text{if $e=\eta$.}\\
  \tau_{e} & \text{if $e\neq \eta$.}
 \end{cases}$$
 Note that for any edge $\eta\in E(\Z^d)$, since $\tau$ and $\tau^{(\eta)}$ have same distributions,
\begin{equation}\label{ship}
  \begin{split}
    &\P\left(\exists \Gamma \in \mathbb{O}_N,~\eta\in\Gamma,~\tau_{\eta}>\alpha\right)\P\left(\tau^*_{\eta}\leq \alpha\right)\\
    &=\P\left(\exists \Gamma \in \mathbb{O}_N,~\eta\in\Gamma,~\tau_{\eta}>\alpha,~\tau^{(\eta)}_{\eta}\leq \alpha \right)\\
    &\leq \P\left(\forall \Gamma \in \mathbb{O}^{(\eta)}_N,~\eta\in\Gamma\right)= \P\left(\forall \Gamma \in \mathbb{O}_N,~\eta\in\Gamma\right),
  \end{split}
\end{equation}
where $\mathbb{O}^{(\eta)}_N$ is the set of all optimal paths from the origin to $N\mathbf{e}_1$ with respect to $\tau^{(\eta)}$. Indeed, if there exists $\Gamma \in \mathbb{O}_N$ such that $\eta\in\Gamma,~\tau_{\eta}>\alpha,~\tau^{(\eta)}_{\eta}\leq \alpha$, since $t^{(\eta)}(0,N\mathbf{e}_1)<t(0,N\mathbf{e}_1)$ and $t^{(\eta)}(\Gamma)<t(\Gamma)$ for any path $\Gamma$ with $e\notin\Gamma$, optimal paths for $\tau^{(\eta)}$ must pass through $\eta$. Therefore the inequality of \eqref{ship} follows.  Thus we have
\begin{equation}\label{impor}
  \begin{split}
    \E\left[\sharp\left[\bigcap_{\Gamma\in\mathbb{O}_N}\Gamma\right]\right]&= \sum_{\eta\in E(\Z^d)}\P\left(\forall \Gamma \in \mathbb{O}_N,~\eta\in\Gamma\right)\\
    &\geq \P\left(\tau^*_{\eta}\leq \alpha\right) \sum_{\eta\in E(\Z^d)}\P\left(\exists \Gamma \in \mathbb{O}_N,~\eta\in\Gamma,~\tau_{\eta}>\alpha\right)\\
    &=F(\alpha)\E\left[\sharp \left\{\eta\in E(\Z^d)|~\exists \Gamma \in \mathbb{O}_N,~\eta\in\Gamma,~\tau_{\eta}>\alpha\right\}\right],
  \end{split}
\end{equation}
where $F(\alpha)=\P(\tau_e\le \alpha)$.
\end{proof}
Next we show that there exist $\alpha>F^-$ and $c>0$ such that for any $N\in\N$
\begin{equation}\label{number}
  \begin{split}
    \E[\sharp\left\{ \eta \in E(\Z^d)|~\exists \Gamma\in\mathbb{O}_N\text{ such that }\eta \in \Gamma,~\tau_{\eta}>{}\alpha\right\}]\geq cN.
\end{split}
 \end{equation}
In fact, if we take $\alpha>F^-$ with $\P(\tau_e>\alpha)>0$ and $\tilde{\tau}_e=\tau_e+\mathbb{I}_{\{\tau_e>\alpha\}}$, the result of \cite{BK} leads us to that
$$\lim_{N\to\infty}N^{-1}\E[t(0,N\mathbf{e}_1)]<\lim_{N\to\infty}N^{-1}\E[\tilde{t}(0,N\mathbf{e}_1)],$$
where $\tilde{t}(\cdot,\cdot)$ is the first passage time with respect to $\tilde{\tau}$. Since $$\tilde{t}(0,N\mathbf{e}_1)\leq t(0,N\mathbf{e}_1)+\sharp\left\{ \eta \in E(\Z^d)|~\exists \Gamma\in\mathbb{O}_N\text{ such that }\eta \in \Gamma,~\tau_{\eta}>{}\alpha\right\},$$
 we have \eqref{number}. Combining it with Lemma~\ref{connection}, we have Theorem~\ref{thm:intersection}.
\section{Proof of Theorem~\ref{thm:union}}
We take $\alpha>F^-$ arbitrary. By \eqref{impor}, we have
\begin{equation}\label{impor2}
  \begin{split}
    &\E\left[\sharp \left\{e\in E(\Z^d)|~\exists \Gamma \in \mathbb{O}_N,~e\in\Gamma,~\tau_e>\alpha\right\}\right]\leq  F(\alpha)^{-1}\E\left[\min\{\sharp\Gamma|~\Gamma \in \mathbb{O}_N\}\right]\leq CN,
  \end{split}
\end{equation}
where we have used \eqref{length-opt} below in the last inequality. The rest of the proof is similar to that of Theorem~2 in \cite{Zhang06} except for Lemma~\ref{00}. We take $\alpha>F^-$ so that $\P(\tau_e>\alpha)>0$. Let $K_{N}=\left\{e\in E(\Z^d)|~\exists \Gamma \in \mathbb{O}_N,~e\in\Gamma,~\tau_e>\alpha\right\}$, and $R_N=\bigcup_{\Gamma\in\mathbb{O}_N}\Gamma$.
By \eqref{impor2}, for sufficiently large $M>0$,
\begin{equation}\label{5.2}
  \begin{split}
    \E\sharp R_N &\leq  \E\left[\sharp R_N; \sharp R_N \geq M\sharp K_N\right] + MCN\\
      &\leq \E[(\sharp R_N)^2]^{1/2}(\P(\sharp R_N \geq M\sharp K_N))^{1/2} + MCN
  \end{split}
\end{equation}
Now we will estimate $\P(\sharp R_N \geq M\sharp K_N)$. Given $k\in\Z$  and $u\in \Z^d$, we define the square $B_k(u)$ whose size is $k$ and corner is $ku$ and the fattened $\hat{R}_N$ by
$$B_k(u)=\prod^d_{i=1}[ku_i,ku_i+ k),$$
  $$\hat{R}_N(k)=\{u\in \Z^d: B_k(u)\cap R_N \neq\emptyset \}.$$
  Note that $\hat{R}_N(k)$ is connected and contains the origin. By our definition,
\begin{equation}\label{5.3}
  \begin{split}
\sharp \hat{R}_N(k)\geq \sharp R_N(k)/k^d.
  \end{split}
\end{equation}
Since $\sharp R_N\geq N$ and \eqref{5.3},
\begin{equation}\label{5.4}
  \begin{split}
    \P(\sharp R_N \geq M\sharp K_N)=\sum_{m\geq{}N/k^d}\P(\sharp R_N \geq M\sharp K_N,~\sharp \hat{R}_N(k) =m).
  \end{split}
  \end{equation}
Denote by $\bar{B}_k(u)$ the vertex set of $B_k(u)$ and all of its neighbor cubes with respect to $|\cdot|_{\infty}$. A cube $B_k(u)$ is said to be bad if $\bar{B}_k(u)\cap K_N\neq\emptyset$ and $u\in\hat{R}_N(k)$. Otherwise, the cube is said to be good. Let $\mathcal{B}_k(u)$ be the event that $B_k(u)$ is bad and $D_N$ be the number of bad cubes $B_k(u)$ for $u\in  \hat{R}_N(k)$. The following lemma corresponds to (5.7) of \cite{Zhang06}. 
\begin{lem}
  On $\{\sharp R_N \geq M\sharp K_N,~\sharp \hat{R}_N(k) =m\}$ for $m\geq{}N/k^d$, if $(8k)^d<M$,
  
\begin{equation}\label{5.6}
  \begin{split}
    D_{N}\geq{}m/2
  \end{split}
\end{equation}
\end{lem}
\begin{proof}
  If there are $m/2$ good cubes, $\sharp K_N\geq{}m/(2\cdot 4^d)\geq 8^{-d}m,$ where $4^d$ appears because of the overlap of cubes. By \eqref{5.3}, on $\{\sharp R_N \geq M\sharp K_N,~\sharp \hat{R}_N(k) =m\}$,
 \begin{equation}\label{5.7}
   \begin{split}
     \sharp R_N\geq M \sharp K_N\geq{}8^{-d}Mm={}8^{-d}M \sharp \hat{R}_N(k)>  \sharp \hat{R}_N k^d,
  \end{split}
 \end{equation}
 which is a contradiction of \eqref{5.3}.
\end{proof}
Thus we take $M>(8k)^d$. Then
\begin{equation}\label{5.9}
  \begin{split}
    \P(\sharp R_N\geq{}M\sharp K_N)&=\sum_{m\geq{}N/k^d}\P(\sharp R_N \geq M\sharp K_N,~\sharp \hat{R}_N(k) =m,~D_N\geq{}m/2)\\
    &= \sum_{m\geq{}N/k^d}\sum_{\kappa_m}\P(\sharp R_N \geq M\sharp K_N,~\hat{R}_N(k) =\kappa_m,~D_N\geq{}m/2),
  \end{split}
\end{equation}
where $\kappa_m$ is a connected subset of $\Z^d$ with $m$ vertices which contains the origin and the second sum is taken over all possible such $\kappa_m$.
\begin{lem}\label{00}
  Suppose that $F$ is useful and there exists $\ell>2(d-1)$ such that $\E[\tau_e^\ell]<\infty$. If $0\notin \bar{B}_k(u)$ and $N\mathbf{e}_1\notin \bar{B}_k(u)$,
 $$\P(\mathcal{B}_k(u))\to 0\hspace{7mm}\text{as $k\to\infty$}.$$
\end{lem}
\begin{proof}
  It suffices to show that there exist constant $C,\epsilon>0$ such that for any  $a,b\in\Z^d$ with $|a-b|_1\geq{}k$,
\begin{equation}\label{zzz}
  \begin{split}
    \P(\text{$\exists$ optimal path $\Gamma$ for $t(a,b)$ such that $\forall e\in\Gamma$, $\tau_e\leq \alpha$})\leq Ck^{-2(d-1)-\epsilon}.
  \end{split}
\end{equation}
Indeed, if the event $\mathcal{B}_k(u)$ occurs, there exist $a\in\{v\in B_k(u)|~\exists w\notin B_k(u)~s.t.~|v-w|_1=1\}$, $b\in\{v\notin \bar{B}_k(u)|~\exists w\in \bar{B}_k(u)~s.t.~|v-w|_1=1\}$ and an optimal path $\Gamma$ for $t(a,b)$ such that $\forall e\in\Gamma$, $\tau_e\leq \alpha$. Since $\sharp[ \{v\in B_k(u)|~\exists w\notin B_k(u)~s.t.~|v-w|_1=1\}\lor\{v\notin \bar{B}_k(u)|~\exists w\in \bar{B}_k(u)~s.t.~|v-w|_1=1\}]\leq C'k^{d-1}$ with some constant  $C'=C'(d)$,
$$\P\left(\mathcal{B}_k(u)\right)\leq (C'k^{d-1})^2 Ck^{-2(d-1)-\epsilon}\to 0\hspace{4mm}\text{as $k\to\infty$}.$$

We will show \eqref{zzz}. From the result  (or the same argument) of \cite{BK}, if we set $\tilde{\tau}_e=\tau_e+\mathbb{I}_{\{\tau_e> \alpha\}}$, there exists a constant $c>0$ independent of $k$ such that $\E[\tilde{t}(a,b)]-\E[t(a,b)]\geq{}ck$. In addition,  if we take $\epsilon>0$ so that $2(d-1)+2\epsilon < \ell$, from Theorem~1 of \cite{Zhang10} with $m=2(d-1)+\epsilon$, there exist $C_1,\tilde{C}_1>0$ such that
\begin{equation}\label{conc}
  \begin{split}
\P(|t(a,b)-\E[t(a,b)]|>ck/4)&\leq (ck/4)^{-2m}\E[(t(a,b)-\E[t(a,b)])^{2m}]\\
&\leq C_1k^{-2(d-1)-2\epsilon}(\log{k})^{14(d-1)+14\epsilon}\leq C_1k^{-2(d-1)-\epsilon},
\end{split}
\end{equation}
and
\begin{equation}\label{conc1}
  \begin{split}
    \P(|\tilde{t}(a,b)-\E[\tilde{t}(a,b)]|>ck/4)\leq  \tilde{C}_1k^{-2(d-1)-\epsilon}.
    \end{split}
\end{equation}
Thus the left hand side of \eqref{zzz} can be bounded from above by
$$\P(\tilde{t}(a,b)=t(a,b))\leq (C_1+\tilde{C}_1)k^{-2(d-1)-\epsilon}.$$
  \end{proof}
Let $L>0$ to be chosen later. The above lemma yields that if we take $k$ sufficiently large, we have $\P(\mathcal{B}_k(u))\leq \exp{(-4L)}$ and in addition, if $m>4^{d+2}$,
\begin{equation}\label{5.10}
  \begin{split}
    \P(\sharp R_N \geq M\sharp K_N,~\hat{R}_N(k) =\kappa_m,~D_N\geq{}m/2)&\leq m\left(\begin{array}{c}
      m\\
      m/2\\
    \end{array}      \right)\exp{(-4L(m/2-2\cdot 4^d))}\\
    &\leq m\left(\begin{array}{c}
      m\\
      m/2\\
    \end{array}      \right)\exp{(-Lm)},
  \end{split}
\end{equation}
where $2\cdot 4^d$ appears because of the condition $0\notin \bar{\mathcal{B}}_k(u)$ and $N\mathbf{e}_1\notin \bar{\mathcal{B}}_k(u)$.\\

The following lemma appears in (4.24) in \cite{Grimm}. We skip the proof.
\begin{lem}
     There exists $C_1>0$ such that for any $m$,
  \begin{equation}\label{conn}
    \begin{split}
       \sharp \{\kappa_m\subset \Z^d|~0\in\kappa_m,~\kappa_m\text{ is connected ,}~|\kappa|=m\}\leq e^{C_1m}.
  \end{split}
\end{equation}
\end{lem}
Thus if we take $L$ sufficiently large and thus $k$ as well, we have
\begin{equation}\label{5.13}
  \begin{split}
    \P(\sharp R_N\geq{}M\sharp K_N)&\leq \sum_{m\geq{}N/k^d}m\left(\begin{array}{c}
      m\\
      m/2\\
    \end{array}      \right)\exp{(-Lm+C_1m)}\\
    &\leq \sum_{m\geq{}N/k^d} \exp{(-Lm/2)}\leq C_2 \exp{(-Lk^{-d}N/2)}, 
  \end{split}
\end{equation}
with some constant $C_2>0$.\\

Let us move on to the estimate of $\E[(\sharp R_N)^2]$.
\begin{lem}\label{exp}
  Suppose that $\E[\tau_e^2]<\infty$ and $F$ is useful. Then there exists $C_3>0$ such that
  $$\E[(\sharp R_N)^2]\leq C_3N^{2d}.$$
\end{lem}
\begin{proof}
 First we suppose that $F^-=0$. From Proposition 5.8 in \cite{Kes86}, there exist $A,B,C>0$ such that for any $r>0$
   \begin{equation}\label{5.8}
     \begin{split}
       \P\left( \exists\text{ selfavoiding path $\Gamma$ from the origin
         with $|\Gamma|\geq r$ and $t(\Gamma) < A r$}\right) < B \exp{(Cr)}.
     \end{split}
     \end{equation}
  We take $K>2\E[\tau_e]/A$. Then for any $s>K$,
  \begin{equation}\label{fin}
    \begin{split}
      &\P\left(\exists \Gamma\in\mathbb{O}_N\text{ such that }|\Gamma|\geq sN\right)\\
      &\leq \P\left(\exists \Gamma\in\mathbb{O}_N\text{ such that }|\Gamma|\geq sN\text{ and }t(0,N\mathbf{e}_1)< AsN\right)+\P\left(t(0,N\mathbf{e}_1)\geq AsN\right)\\
      &\leq B\exp{(-CsN)}+\P\left(t(0,N\mathbf{e}_1)\geq AsN\right).
  \end{split}
  \end{equation}
  where we have used \eqref{5.8} in the second inequality. Now we consider $2d$ disjoint paths from the origin to $N\mathbf{e}_1$ so that $\max\{|r_i||~i=1,\cdots,2d\}\leq N+8$ as in \cite[p 135]{Kes86}. By the Chebyshev inequality, we have that there exists $D=D(d,F,A)>0$ such that  
   \begin{equation}\label{fell}
     \begin{split}
       \P\left(t(0,N\mathbf{e}_1)\geq AsN\right)& \leq \prod^{2d}_{i=1} \P\left(t(r_i)\geq AsN\right)\\
       &\leq \prod^{2d}_{i=1} \P\left(|t(r_i)-\E[t(r_i)]|\geq AsN/2\right)\\
       &\leq \prod^{2d}_{i=1} \left((AsN/2)^{-2}(N+8) \E[\tau_e^2]\right)\leq Ds^{-4d}N^{-2d}. 
  \end{split}
   \end{equation}
   Thus we have for $s>K$,
  \begin{equation}\label{length-opt}
    \P\left(\exists \Gamma\in\mathbb{O}_N\text{ such that }|\Gamma|\geq sN\right) \leq 2Ds^{-4d}N^{-2d}.
    \end{equation}
   Since $\sharp R_N\leq 2d\left(\max_{\Gamma\in\mathbb{O}_N}\sharp\Gamma\right)^d$, there exists $C_3>0$ such that
    \begin{equation}\label{fel}
    \begin{split}
      \E\left[\left(\sharp R_N\right)^2\right]&\leq (2d)^2\E\left[\left(\max_{\Gamma\in\mathbb{O}_N}\sharp\Gamma\right)^{2d}\right]=  \left(2d\right)^{3}\int^\infty_{0}r^{2d-1}\P\left(\max_{\Gamma\in\mathbb{O}_N}\sharp\Gamma\geq r\right)dr\\
      &\leq \left(2d\right)^3 \left((KN)^{2d}+\int^\infty_{KN}r^{2d-1}\P\left(\max_{\Gamma\in\mathbb{O}_N}\sharp\Gamma\geq r\right)dr\right)\\
      &\leq \left(2d\right)^3 \left((KN)^{2d}+\int^\infty_{KN}r^{2d-1}\cdot 2D(r/N)^{-4d}N^{-2d}dr\right)\leq C_3N^{2d}.\\
  \end{split}
    \end{equation}
    
  When $F^->0$, since $\max_{\Gamma\in\mathbb{O}_N}\sharp\Gamma\leq t(0,N\mathbf{e}_1)/F^-,$ the proof is completed as before.
  \end{proof}
Comibining \eqref{5.2}, \eqref{5.13} and the above lemma, we complete the proof.

\section{Proof of Corollary \ref{cor1}.}
\begin{proof}[Proof of Corollary \ref{cor1}]
  As in the proof of Lemma~\ref{exp}, if we take $K>0$ as in the proof of Lemma~\ref{exp}, 
\begin{equation*}
  \begin{split}
    cN&\leq  \E\left[\sharp \left[\bigcap_{\Gamma\in\mathbb{O}_N}\Gamma\right]\right]\\
    &\le \E\left[\left(\sharp \left[\bigcap_{\Gamma\in\mathbb{O}_N}\Gamma\right]\right)^2\right]\P\left(\sharp \left[\bigcap_{\Gamma\in\mathbb{O}_N}\Gamma\right]> KN\right) +KN\P\left(cN/2\leq \sharp \left[\bigcap_{\Gamma\in\mathbb{O}_N}\Gamma\right]\leq KN\right)+cN/2\\
    & \le 1+KN\P\left(cN/2\leq\sharp \left[\bigcap_{\Gamma\in\mathbb{O}_N}\Gamma\right]\right)+cN/2.
\end{split}
 \end{equation*}
  Rearranging this, we have the conclusion.
\end{proof}
\section*{Acknowledgements}
The author would like to express his gratitude to Masato Takei for introducing him the idea of Theorem~2 in \cite{Zhang06}. He is also indebted to Hugo Duminil--Copin for introducing \cite{CKNPS} prior to its publication. This research is partially supported by JSPS KAKENHI 16J04042.

\end{document}